\documentclass[11pt]{article}
\pdfoutput=1

\usepackage[margin=1.3in]{geometry}

\usepackage{graphicx}
\usepackage{amssymb,amsmath}
\usepackage{amsthm,enumerate}
\usepackage{hyperref,xcolor}
\usepackage{mathrsfs}
\usepackage{extarrows}
\usepackage{textcomp}
\allowdisplaybreaks[4]

\makeatletter

\newdimen\bibspace
\makeatother

\numberwithin{equation}{section}

\newtheorem{theorem}{Theorem}[section]
\newtheorem{lemma}[theorem]{Lemma}
\newtheorem{proposition}[theorem]{Proposition}

\newtheorem{remark}[theorem]{Remark}

\def\<{\langle}
\def\>{\rangle}

\def\XXint#1#2#3{{\setbox0=\hbox{$#1{#2#3}{\int}$ }
\vcenter{\hbox{$#2#3$ }}\kern-.6\wd0}}

\begin{document}

\title{Asymptotic expansion of 2-dimensional gradient graph with vanishing mean curvature at infinity}

\author{Zixiao Liu and Jiguang Bao\footnote{
 Supported by the National Key Research and Development Program of China (No. 2020YFA0712904) and
 National Natural Science Foundation of China (No. 11871102 and No. 11631002).}}
\date{\today}

\maketitle

\begin{abstract}
In this paper, we establish the asymptotic expansion at infinity of gradient graph in dimension 2 with vanishing mean curvature at infinity. This corresponds to  our previous results in higher dimensions and generalizes the results for minimal gradient graph on exterior domain in dimension 2. Different from the strategies for higher dimensions, instead of the equivalence of Green's function on unbounded domains,
we apply a version of iteration methods from Bao--Li--Zhang [Calc.Var PDE, 52(2015), pp. 39-63] that is refined by spherical harmonic expansions to provide a more explicit asymptotic behavior than known results.
 \\[1mm]
 \\[1mm]
\end{abstract}

\section{Introduction}

We consider the asymptotic expansion at infinity of solutions to a family of mean curvature equations of gradient graph in dimension $2$.

Let  $(x,Du(x))$ denote the gradient graph of $u$ in  $(\mathbb R^n\times\mathbb R^n, g_{\tau})$, where $Du$ denotes the gradient of scalar function $u$ and

\begin{equation*}
g_{\tau}=\sin \tau \delta_{0}+\cos \tau g_{0}, \quad \tau \in\left[0, \frac{\pi}{2}\right]
\end{equation*}
is the linearly combined metric of standard Euclidean metric

\begin{equation*}
\delta_{0}=\sum_{i=1}^{n} d x_{i} \otimes d x_{i}+\sum_{j=1}^{n} d y_{j} \otimes d y_{j},
\end{equation*}
with the pseudo-Euclidean metric

\begin{equation*}
g_{0}=\sum_{i=1}^{n} d x_{i} \otimes d y_{i}+ \sum_{j=1}^{n} d y_{j} \otimes d x_{j}.
\end{equation*}
As proved in \cite{Chong-Rongli-Bao-SecondBoundary-SPL},  if $u\in C^2(\mathbb R^n)$ is a classical solution of
\begin{equation}\label{Equ:intro}
F_{\tau}(\lambda(D^2u))=f(x),
\end{equation}
then $Df(x)$ is the mean curvature of gradient graph $(x,Du(x))$ in $\left(\mathbb{R}^{n} \times \mathbb{R}^{n}, g_{\tau}\right)$.
In \eqref{Equ:intro}, $f(x)$ is a sufficiently regular function,  $\lambda\left(D^{2} u\right)=\left(\lambda_{1}, \lambda_{2}, \cdots, \lambda_{n}\right)$ is the vector formed by $n$ eigenvalues of Hessian matrix $D^{2} u$ and

$$
F_{\tau}(\lambda):=\left\{
\begin{array}{ccc}
\displaystyle  \frac{1}{n} \sum_{i=1}^{n} \ln \lambda_{i}, & \tau=0,\\
\displaystyle  \frac{\sqrt{a^{2}+1}}{2 b} \sum_{i=1}^{n} \ln \frac{\lambda_{i}+a-b}{\lambda_{i}+a+b},
  & 0<\tau<\frac{\pi}{4},\\
  \displaystyle-\sqrt{2} \sum_{i=1}^{n} \frac{1}{1+\lambda_{i}}, & \tau=\frac{\pi}{4},\\
  \displaystyle\frac{\sqrt{a^{2}+1}}{b} \sum_{i=1}^{n} \arctan \displaystyle\frac{\lambda_{i}+a-b}{\lambda_{i}+a+b}, &
  \frac{\pi}{4}<\tau<\frac{\pi}{2},\\
  \displaystyle\sum_{i=1}^{n} \arctan \lambda_{i}, & \tau=\frac{\pi}{2},\\
\end{array}
\right.
$$
$a=\cot \tau, b=\sqrt{\left|\cot ^{2} \tau-1\right|}$.

When $\tau=0$, Eq.~\eqref{Equ:intro} becomes the Monge--Amp\`ere type equation

$$
\det D^2u=e^{nf(x)}.
$$
When $\tau=\frac{\pi}{4}$, Eq.~\eqref{Equ:intro} can be translated into the inverse harmonic Hessian equation

$$
-\sqrt 2\sum_{i=1}^n\dfrac{1}{\lambda_i}=f(x),
$$
which is a special form of the quadratic Hessian equations $\frac{\sigma_k(\lambda)}{\sigma_l(\lambda)}=f(x)$, where $\sigma_k(\lambda)$ with $k=1,2,\cdots,n$ denotes the $k$-th elementary symmetric function of $\lambda(D^2u)$.\\
\noindent When $\tau=\frac{\pi}{2}$, Eq.~\eqref{Equ:intro} becomes the Lagrangian mean curvature type equation
\begin{equation}\label{equ:SPL}
\sum_{i=1}^{n} \arctan \lambda_{i}\left(D^{2} u\right)=f(x).
\end{equation}
Especially when $f(x)$ is a constant, the Lagrangian mean curvature equation above is also known as the special Lagrangian equation.

For $f(x)$ being a constant $C_0$, Warren \cite{Warren-Calibrations-MA} proved Bernstein-type results of \eqref{Equ:intro} based on the results of J\"orgens \cite{Jorgens}--Calabi \cite{Calabi}--Pogorelov \cite{Pogorelov}, Flanders \cite{Flanders} and Yuan \cite{Yuan-Bernstein-SPL,Yuan-GlobalSolution-SPL}, which state that any classical solution with under suitable semi-convex conditions must be a quadratic. Especially when $\tau=0$, there are different proofs and extensions of the Bernstein-type results of Monge--Amp\`ere equations by Cheng--Yau \cite{ChengandYau}, Caffarelli \cite{Caffarelli-InteriorEstimates-MA}, Jost--Xin \cite{JostandXin}, Fu \cite{Fu-Bernstein-SPL}, Li--Xu--Simon--Jia \cite{Book-Li-Xu-Simon-Jia-MA}, etc. When $\tau=\frac{\pi}{4}$, there are Bernstein-type results of Hessian quotient equations by Bao--Chen--Guan--Ji \cite{Bao-Chen-Guan-Ji--HessianQuotient}. For generalizations of Bernstein-type results of Hessian and Hessian quotient equations, we refer to Chang--Yuan \cite{Chang-Yuan-LiouvillSigma2}, Chen--Xiang \cite{Chen-Xiang-Rigidity2Hessian},
Li--Ren--Wang \cite{Li-Ren-Wang-InteriRigidty}, Yuan \cite{Yuan-Bernstein-SPL}, Du \cite{Du-NecessaryandSufficient-HessianQuotient} etc.
When $\tau=\frac{\pi}{2}$, though the mean curvature relies only on $Df$, Yuan \cite{Yuan-GlobalSolution-SPL} reveals the importance of the value of phase $C_0$. The phase (also known as the Lagrangian angle) $\frac{n-2}{2}\pi$ is called critical since the level set

$$
\left\{\lambda\in\mathbb{R}^n~|~\sum_{i=1}^n\arctan\lambda_i=C_0\right\}
$$
is convex only when $|C_0|\geq \frac{n-2}{2}\pi$. Another way to obtain a convexity/concaviety structure is to restrict in the range of $D^2u>0$ as in Yuan \cite{Yuan-Bernstein-SPL}.
For further relevant discussions, we refer to
Warren--Yuan \cite{Warren-Yuan-HessianEstimate-Sigma2,Warren-Yuan-LargePhase}, Wang--Yuan \cite{Wang-Yuan-SuperCritical},
Chen--Shankar--Yuan \cite{Chen-Shankar-Yuan-RegularitySPL},
 Li--Li--Yuan \cite{Li-Li-Yuan-BernsteinThm}, Bhattacharya--Shankar \cite{Bhattacharya-Shankar-OptimalRegularity-LagMeanCur,Bhattacharya-Shankar-RegularityLagrangiaMC},
Bhattacharya \cite{Bhattacharya-HessianEstiamte-LagMeanCur} and the references therein.

For $f(x)-C_0$ having compact support and $n\geq 3$, there are exterior Bernstein type results of \eqref{Equ:intro} by the authors \cite{Liu-Bao-2020-ExpansionSPL}, which state that any classical solution with the same semi-convex conditions must be asymptotic to quadratic  polynomial at infinity, together with higher order expansions that give the precise gap between exterior minimal gradient graph and the entire case.
For $f(x)-C_0$ having compact support and $n=2$, the authors \cite{Liu-Bao-2021-Dim2} proved similar exterior Bernstein type results of \eqref{Equ:intro}, which imply that exterior solutions are asymptotic to quadratic polynomial with additional $\ln$-term at infinity. When $\tau=0$, such results were partially proved earlier for Monge--Amp\`ere equations by Caffarelli--Li \cite{Caffarelli-Li-ExtensionJCP} and Hong \cite{Hong-Remark-MA}.
When $\tau=\frac{\pi}{2}$, such  results were partially  proved earlier for special Lagrangian equations by Li--Li--Yuan \cite{Li-Li-Yuan-BernsteinThm}. The refined asymptotic expansions in our earlier results
\cite{Liu-Bao-2020-ExpansionSPL,Liu-Bao-2021-Dim2} are new
 even for $\tau=0$ and $\frac{\pi}{2}$ cases, it reveals that the gap between exterior minimal gradient graph and the entire case can be written into  higher order errors.

For $f(x)-C_0$ vanishing at infinity,  $n\geq 3$ and $\tau\in[0,\frac{\pi}{4}]$,  there are  exterior Bernstein-type results of \eqref{Equ:intro}
by the authors \cite{Liu-Bao-2021-Expansion-LagMeanC}, which provide both the asymptotic behavior and finer expansions of error terms. Especially when $\tau=0$, the asymptotic behavior result of solutions to Monge--Amp\`ere type equations were proved under stronger assumptions on $f(x)$ by Bao--Li--Zhang \cite{Bao-Li-Zhang-ExteriorBerns-MA}.

Such asymptotic expansion  for solutions of geometric curvature equations were earlier introduced in Han--Li--Li \cite{Han-Li-Li-AsymExpan-Yamabe}, which refines the previous study on the Yamabe equation and the $\sigma_k$-Yamabe equations by Caffarelli--Gidas--Spruck \cite{Caffarelli-Gidas-Spruck-CGS}, Korevaar--Mazzeo--Pacard--Schoen \cite{KMPS-RefinedAsymptotics}, Han--Li--Teixeira \cite{Han-Li-Teixeira-AsymBeha-kYamabe}, etc.  We would also like to mention that for the Monge--Amp\`ere type equations, there are also classification results and asymptotic behavior analysis for $f(x)-C_0$ being a periodic function by Caffarelli--Li \cite{Caffarelli-Li-Liouville-MA-periodic} or asymptotically periodic at infinity by Teixeira--Zhang \cite{Teixeira-Zhang-Liouville-MA-AsymptoticPeriodic}, etc. Under additional assumptions on $D^2u$ at infinity, the asymptotic behavior results were obtained for general fully nonlinear elliptic equations by
Jia \cite{Jia-Xiaobiao-AsymGeneralFully}.

In this paper, we consider the asymptotic behavior and further
expansions or error terms
at infinity of solutions to \eqref{Equ:intro} with $f(x)-C_0$ vanishing at infinity and $n=2$. Some special structures in $n=2$ case enable us to deal with all $\tau\in[0,\frac{\pi}{2}]$, which is different from $n\geq 3$ case in \cite{Liu-Bao-2021-Expansion-LagMeanC}.  However, there are also disadvantages caused by $n=2$, especially the lack of equivalence on the Green's function on unbounded domain.
The asymptotic behavior obtained here is a refinement of known results of the Monge--Amp\`ere equations by Bao--Li--Zhang \cite{Bao-Li-Zhang-ExteriorBerns-MA}.

Consider classical solutions of
\begin{equation}\label{Equ-SPL-perturb}
  F_{\tau}\left(\lambda\left(D^{2} u\right)\right)=f(x)\quad\text{in }\mathbb R^2,
\end{equation}
where $f(x)\in C^m(\mathbb R^2)$ converge to some constant $f(\infty)$ in the sense of
\begin{equation}\label{Low-Regular-Condition}
  \limsup _ { | x | \rightarrow \infty } | x | ^ {\zeta+k} | D^k( f ( x ) - f ( \infty ) ) | < \infty,\quad\forall~ k=0,1,2,\cdots,m
\end{equation}
for some $\zeta>2$ and $m\geq 3$.

From the definition of $F_{\tau}$ operator, $\lambda_i$ must satisfy

$$
\left\{
\begin{array}{llll}
  \lambda_i>0, & \text{for }\tau=0,\\
  \frac{\lambda_i+a-b}{\lambda_i+a+b}>0, & \text{for }\tau\in(0,\frac{\pi}{4}),\\
  \lambda_i\not=-1, & \text{for }\tau=\frac{\pi}{4},\\
  \lambda_i+a+b\not=0, & \text{for }\tau\in(\frac{\pi}{4},\frac{\pi}{2}),\\
\end{array}
\right.
$$
for $i=1,2$. Thus we separate the solution into semi-convex and semi-concave cases.
For simplicity, we consider the semi-convex case
\begin{equation}\label{equ:cond:semi-convex}
  A>\left\{
  \begin{array}{llll}
    0, & \tau=0,\\
    -(a-b)I, & \tau \in\left(0, \frac{\pi}{4}\right),\\
    -I, & \tau=\frac{\pi}{4},\\
    -(a+b)I, & \tau \in\left(\frac{\pi}{4}, \frac{\pi}{2}\right),\\
    -\infty, & \tau=\frac{\pi}{2},\\
  \end{array}
  \right.
\end{equation}
where $I$ denotes the 2-by-2 identity matrix and the semi-concave case can be treated similarly.

Hereinafter, we assume $D^2u$ satisfy \eqref{equ:cond:semi-convex} in $\mathbb R^2$,
\begin{equation}\label{equ:cond:fInfinity}
f(\infty)\not=0\quad\text{for }\tau=\frac{\pi}{2}\quad
\text{and }
f(\infty)\not\in\partial\left\{
F_{\tau}(\lambda(A))~|~A \text{ satisfies }\eqref{equ:cond:semi-convex}
  \right\},
\end{equation}
 where the notation $\partial$ denote the boundary of a set in $\mathbb R$.
We may assume further without loss of generality that $f(\infty)>0$ for $\tau=\frac{\pi}{2}$ case, otherwise consider $-u$ instead.

\begin{remark}
  In condition \eqref{equ:cond:fInfinity}, our major additional restriction is
$f(\infty)\not=0$ for $\tau=\frac{\pi}{2}.$
  It corresponds to the critical phase in $\mathbb R^2$, which leads to a different phenomenon than supercritical case.

  If $\tau\in(0,\frac{\pi}{4}]$ and  $\displaystyle f(\infty)=\sup\{F_{\tau}(A)~|~A\text{ satisfies }\eqref{equ:cond:semi-convex}\}=0$, then by the structure of $F_{\tau}(\lambda)$, we have

  $$
  \lambda_1(D^2u(x)), \lambda_2(D^2u(x))\rightarrow+\infty\quad\text{as }|x|\rightarrow \infty.
  $$
  These are not the asymptotic behavior under discussion and hence we rule out these situations by \eqref{equ:cond:fInfinity}.
\end{remark}

Let $\mathtt{Sym}(2)$ denote the set of 2-by-2 symmetric matrix and $x^T$ denote the transpose of a vector $x\in\mathbb R^2$. We
 say a scalar function $\varphi=O_{l}\left(|x|^{-k_{1}}(\ln |x|)^{k_{2}}\right)$ with $l \in \mathbb{N}, k_{1}, k_{2} \geq 0$ if it satisfies

\begin{equation*}
\left|D^{k} \varphi\right|=O\left(|x|^{-k_{1}-k}(\ln |x|)^{k_{2}}\right) \quad \text {as }|x| \rightarrow\infty
\end{equation*}
for all $k=0,1,2,\cdots,l$.

Our main result shows the following asymptotic behavior and expansion result at infinity.
\begin{theorem}\label{thm:perturb-dim2}
  Let $u\in C^2(\mathbb R^2)$ be a classical solution of \eqref{Equ-SPL-perturb} with $D^2u$ satisfying \eqref{equ:cond:semi-convex} and $f\in C^m(\mathbb R^2)$ satisfy \eqref{Low-Regular-Condition}, \eqref{equ:cond:fInfinity} for some $\zeta>2$ and $m\geq 3$.
Assume further that
\begin{equation}\label{equ:cond:quadratic-growth}
   u(x)\leq C(1+|x|^2)\quad\text{in }\mathbb R^2
\end{equation}
for some $C>0$. Then there exist $A\in\mathtt{Sym}(2)$ satisfying $F_{\tau}(\lambda(A))=f(\infty)$ and \eqref{equ:cond:semi-convex}, $b\in\mathbb R^2, c,d\in\mathbb R$
 such that
\begin{equation}\label{equ:asymptotic-behav}
\begin{array}{lllll}
&\displaystyle u(x)-\left(\frac{1}{2} x^{T} A x+b x+c\right)-d \ln \left(x^{T} P x\right)\\
= & \displaystyle
\left\{
\begin{array}{lllll}
O_{m+1}\left(|x|^{2-\min\{\zeta,3\}}\right), & \text{if }\zeta\not=3,\\
O_{m+1}\left(|x|^{-1}(\ln|x|)\right), & \text{if }\zeta=3,\\
\end{array}
\right.\\
\end{array}
\end{equation}
as $|x|\rightarrow\infty$, where the matrix $P$ is given by
\begin{equation}\label{equ:def-matrixQ}
  P=(DF_{\tau}(\lambda(A)))^{-1}=\frac{1}{2}\left(\sin\tau A^2+2\cos\tau A+\sin\tau I\right).
\end{equation}
Furthermore, when $\zeta>3$, there  also exist $d_1,d_2\in\mathbb R$ such that
\begin{equation}\label{equ:asymptotic-expan}
  \begin{array}{llll}
  &\displaystyle u(x)-\left(\frac{1}{2}x^TAx+\beta x+c\right)- d\ln (x^TPx)\\
  =&\displaystyle (x^TPx)^{-\frac{1}{2}}(d_1\cos\theta+d_2\sin\theta)+\left\{
  \begin{array}{llll}
  O_m(|x|^{2-\zeta}), & \text{if }\zeta<4,\\
  O_m(|x|^{-2}(\ln|x|)), & \text{if }\zeta\geq 4,\\
  \end{array}
  \right.\\
  \end{array}
  \end{equation}
as $|x|\rightarrow\infty$, where $\theta=\frac{P^{\frac{1}{2}}x}{(x^TPx)^{\frac{1}{2}}}$.
\end{theorem}
\begin{remark}
  For $\tau=0$ case in Theorem \ref{thm:perturb-dim2}, condition \eqref{equ:cond:quadratic-growth} is not necessary.
\end{remark}

\begin{remark}
Theorem \ref{thm:perturb-dim2} generalizes the asymptotic expansion results of  previous work \cite{Liu-Bao-2021-Dim2} by the authors, where $f(x)$ being a constant since it corresponds to $\zeta=\infty$ and $m=\infty$ case in   \eqref{equ:asymptotic-behav} and \eqref{equ:asymptotic-expan}.
But there are many differences in argumentation methods.
By differentiating the equations we only obtain nonhomogeneous elliptic equations and inequalities on exterior domain. Furthermore, when $\frac{b}{\sqrt{a^2+1}}f(x)+\frac{\pi}{2}\equiv 0$ and $\tau\in(\frac{\pi}{4},\frac{\pi}{2})$, the equation can be translated into harmonic equations. But only if $\frac{b}{\sqrt{a^2+1}}f(\infty)+\frac{\pi}{2}=0$, it yields an additional perturbation term involving the second order derivatives of $u$. This leads to the difficult discussions as in \eqref{equ-temp-9}. For a similar reason, $\tau=\frac{\pi}{2}$ case cannot be deduced into the Monge--Amp\`ere equation $\det D^2v=1$ or harmonic equations by a simple change of variable as in \cite{Liu-Bao-2021-Dim2}. For these perturbed cases, we turn to study the algebraic form as in \eqref{equ:transformed} and apply iteration methods instead of using the asymptotic behavior of solutions to the Monge--Amp\`ere equations directly.
\end{remark}

\begin{remark}
As in the discussions in \cite{Bao-Li-Zhang-ExteriorBerns-MA,Liu-Bao-2020-ExpansionSPL,Liu-Bao-2021-Expansion-LagMeanC}
  etc., $\zeta>2$ in \eqref{Low-Regular-Condition} is optimal
  in the sense that for $\zeta=2$ we may construct radially symmetric solutions with $u=\frac{1}{2}x^TAx+O((\ln|x|)^2)$ as $|x|\rightarrow \infty$.
  Furthermore, the asymptotic expansion \eqref{equ:asymptotic-expan} is optimal in the sense that the next order term in \eqref{equ:asymptotic-expan} may contain error terms like $|x|^{-2}\ln|x|$, which cannot be represented into
  $(x^TPx)^{-1}(d_3\cos 2\theta+d_4\sin 2\theta)$ for some $d_3,d_4\in\mathbb R$.
\end{remark}

\begin{remark}
  By extension results as in Theorem 3.2 of \cite{Min-Extension-ConvexFunc},
  we may change the value of $u$ and $f$ on a dense subset without affecting the asymptotic behavior near infinity. Consequently by interior estimates as in  Lemma 17.16 of \cite{Book-Gilbarg-Trudinger},
  the regularity assumption on $f$ can be relaxed to $f\in C^0(\mathbb R^2)$ with $D^mf$ exists outside a compact subset of $\mathbb R^2$ for some $m\geq 3$.  Especially since $m\geq 3$, we may assume without loss of generality that $u\in C^{4}(\mathbb R^2)$.
\end{remark}

The paper is organized as follows.
In section \ref{seclabel-Poisson} we prove existence results for Poisson equations on exterior domain of $\mathbb R^2$.
 In section \ref{seclabel-WeakConvergence} we prove that  $u$ converge to a quadratic function $\frac{1}{2}x^TAx$ at infinity with a speed of $O(|x|^{2-\epsilon})$ for some $\epsilon>0$, which is similar to the strategy used in \cite{Caffarelli-Li-ExtensionJCP,Bao-Li-Zhang-ExteriorBerns-MA} etc. In section \ref{seclabel-AsyExpan} we prove Theorem \ref{thm:perturb-dim2} by iteration and spherical harmonic decomposition, based on the results in sections \ref{seclabel-Poisson} and \ref{seclabel-WeakConvergence}.

\section{Preliminary results on Poisson equations}\label{seclabel-Poisson}

In this section, we prove an existence result for Poisson equation on exterior domain.
\begin{lemma}\label{Lem-fastConverge}
  Let $g\in C^{\infty}(\mathbb R^2)$ satisfy
\begin{equation}\label{equ:vanishing-g-Lp}
\|g(r \cdot)\|_{L^{p}\left(\mathbb{S}^{1}\right)} \leq c_{0} r^{-k_{1}}(\ln r)^{k_{2}} \quad \forall ~r>1
\end{equation}
for some $c_0>0, k_1>0, k_2\geq 0$ and $p\geq 2$. Then there exists a smooth solution $v$ of
\begin{equation}\label{equ:Laplace}
  \Delta v=g\quad\text{in }\mathbb R^2\setminus\overline{B_1}
\end{equation}
such that
\begin{equation}\label{equ:FastVanishing}
|v(x)| \leq \left\{
\begin{array}{lllll}
Cc_0|x|^{2-k_1}(\ln|x|)^{k_2}, & \text{if }k_1\not\in\mathbb N_*,\\
Cc_0|x|^{2-k_1}(\ln|x|)^{k_2+1}, & \text{if }k_1\in\mathbb N_*\setminus\{2\},\\
Cc_0(\ln|x|)^{k_2+2}, & \text{if }k_1=2,\\
\end{array}
\right.
\end{equation}
for $|x|>1$ for some $C>0$.
\end{lemma}
For $k_1>2$ case, Lemma \ref{Lem-fastConverge} is similar to the one proved  earlier  by the authors in \cite{Liu-Bao-2021-Dim2}.
The proof here is similar with minor modifications.
\begin{proof}[Proof of Lemma \ref{Lem-fastConverge}] Here we only provide detail proof for $0<k_1<1$ case, the rest parts follow with minor modifications on the choice of $a_{k,m}$ and $\epsilon>0$ below as in \eqref{equ-temp-10}.

In polar coordinate we have

$$
\Delta v=\dfrac{\partial^2v}{\partial r^2}+\frac{1}{r}\frac{\partial v}{\partial r}+\frac{1}{r^{2}} \frac{\partial^{2} v}{\partial \theta^{2}},
$$
where $r:=|x|$ represents the radial distance and $\theta:=\frac{x}{|x|}$ the angle.
Let

$$
Y_1^{(0)}(\theta)\equiv \dfrac{1}{\sqrt{2\pi}},\quad Y_1^{(k)}(\theta)=\dfrac{1}{\sqrt \pi}\cos k\theta\quad\text{and}\quad
Y_2^{(k)}(\theta)=\dfrac{1}{\sqrt \pi}\sin k\theta,
$$
which forms a complete  standard orthogonal basis of $L^2(\mathbb S^1)$.
Decompose $g$ and the wanted solution $v$
into

$$
v(x)=a_{0,1}(r)+\sum_{k=1}^{+\infty}\sum_{m=1}^{2} a_{k, m}(r) Y_{m}^{(k)}(\theta),\quad
g(x)=b_{0,1}(r)+\sum_{k=1}^{+\infty}\sum_{m=1}^{2} b_{k, m}(r) Y_{m}^{(k)}(\theta),
$$
where

$$
a_{k,m}(r):=\int_{\mathbb{S}^{n-1}} v(r \theta) \cdot Y_{m}^{(k)}(\theta) d  \theta,\quad b_{k,m}(r):=\int_{\mathbb{S}^{n-1}} g(r\theta) \cdot Y_{m}^{(k)}(\theta)  d  \theta.$$
By the linear independence of $Y_m^{(k)}(\theta)$, \eqref{equ:Laplace} implies that

$$
a_{0,1}''(r)+\frac{1}{r}a_{0,1}'(r)=b_{0,1}(r)\quad\text{in }r>1
$$
and for all $k\in\mathbb{N}_*$ with $m=1,2,$

$$
a_{k,m}^{\prime \prime}(r)+\frac{1}{r} a_{ k,m}^{\prime}(r)-\frac{k^2}{r^{2}} a_{k,m}(r) =b_{k,m}(r)\quad\text{in }r>1.
$$

By solving the ODE, there exist constants $C_{k,m}^{(1)},C_{k,m}^{(2)}$ such that for all $r>1$,
\begin{equation}\label{Equ-def-Wronski}
\begin{array}{lll}
  a_{k,m}(r)&=&C_{k,m}^{(1)}r^k+
C_{k,m}^{(2)}r^{-k}\\
&&\displaystyle
-\dfrac{1}{2k}r^k\int_{2}^r\tau^{1-k}b_{k,m}(\tau) d  \tau
+\dfrac{1}{2k}r^{-k}\int_{2}^r\tau^{1+k}b_{k,m}(\tau) d  \tau
\end{array}
\end{equation}
for $k\geq 1$
and

$$
\begin{array}{lllll}
a_{0,1}(r)&=& C_{0,1}^{(1)}+C_{0,1}^{(2)}\ln r\\
&&\displaystyle -\int_2^r \tau\ln\tau  b_{0,1}(\tau)d \tau+\ln r\int_2^r \tau b_{0,1}(\tau) d \tau.
\end{array}
$$
 By \eqref{equ:vanishing-g-Lp},
\begin{equation}\label{Equ-converge}
|b_{0,1}(r)|^2+\sum_{k=1}^{+\infty}\sum_{m=1}^{2}|b_{k,m}(r)|^2=||g(r\cdot)||^2_{L^2(\mathbb{S}^{n-1})}
\leq c_0^2(2\pi)^{\frac{p-2}{p}}r^{-2k_1}(\ln r)^{2k_2}
\end{equation}
for all $r>1$. Then by $0<k_1<1$, we have $r^{1-k}b_{k,m}(r)\in L^1(2,+\infty)$  for all $k\geq 2$ and $r^{k+1}b_{k,m}(r)\not\in L^1(2,+\infty)$  for all
$k\in\mathbb N$. We choose $C_{k,m}^{(1)}$ and $C_{k,m}^{(2)}$ in \eqref{Equ-def-Wronski} such that

$$
  a_{0,1}(r):=-\int_{2}^r\tau\ln \tau b_{0,1}(\tau)d\tau+\ln r\int_{2}^r\tau b_{0,1}(\tau)d\tau,
$$

$$
a_{k,m} (r):=
- \dfrac{1}{2}r\int_{2}^rb_{k,m}(\tau) d  \tau
+ \dfrac{1}{2}r^{-1}\int_{2}^r\tau^{2}b_{k,m}(\tau) d  \tau
$$
for $k=1$ and

$$
a_{k,m} (r):=
- \dfrac{1}{2k}r^k\int_{+\infty}^r\tau^{1-k}b_{k,m}(\tau) d  \tau
+ \dfrac{1}{2k}r^{-k}\int_{2}^r\tau^{1+k}b_{k,m}(\tau) d  \tau
$$
for all $k\geq 2$.

For $a_{0,1}(r)$, we notice that there are cancellation properties as below. By \eqref{Equ-converge} we have
\begin{equation}\label{equ-cancellation}
  \begin{array}{llll}
  |a_{0,1}(r)| & = & \displaystyle \left|\int_2^r\ln\frac{r}{\tau}\tau b_{0,1}(\tau) d\tau\right|\\
  &\leq & \displaystyle Cc_0\int_2^r\ln\frac{r}{\tau}\tau^{1-k_1}(\ln \tau)^{k_2}d\tau\\
  &\leq &  \displaystyle Cc_0
  \left(
  \ln r\int_2^r\tau^{1-k_1}(\ln \tau)^{k_2}d\tau
  -\int_2^r\tau^{1-k_1}(\ln \tau)^{k_2+1}d\tau
  \right)
  .\\
  \end{array}
\end{equation}
By a direct computation, for any $0<k_1<2$ we have

$$
\begin{array}{llll}
&\displaystyle\int_2^r\tau^{1-k_1}(\ln\tau)^{k_2+1}d\tau\\
=&\displaystyle \dfrac{1}{2-k_1}
  \left(
  r^{2-k_1}(\ln r)^{k_2+1}-2^{2-k_1}(\ln 2)^{k_2+1}-(k_2+1)\int_2^r\tau^{1-k_1}(\ln\tau)^{k_2}d\tau
  \right).
\end{array}
$$
Consequently

$$
\begin{array}{llll}
  & \displaystyle  \ln r\int_2^r\tau^{1-k_1}(\ln \tau)^{k_2}d\tau
  -\int_2^r\tau^{1-k_1}(\ln \tau)^{k_2+1}d\tau\\
  =& \displaystyle \dfrac{1}{2-k_1}
  \left(
  r^{2-k_1}(\ln r)^{k_2+1}-2^{2-k_1}(\ln 2)^{k_2}(\ln r)-k_2\ln r\int_2^r\tau^{1-k_1}(\ln\tau)^{k_2-1}d\tau
  \right)\\
  &-\displaystyle \dfrac{1}{2-k_1}
  \left(
  r^{2-k_1}(\ln r)^{k_2+1}-2^{2-k_1}(\ln 2)^{k_2+1}-(k_2+1)\int_2^r\tau^{1-k_1}(\ln\tau)^{k_2}d\tau
  \right)\\
  \leq& \displaystyle \frac{k_2+1}{2-k_1}\int_2^r\tau^{1-k_1}(\ln\tau)^{k_2}d\tau-\frac{k_2}{2-k_1}\ln r\int_2^r\tau^{1-k_1}(\ln\tau)^{k_2-1}d\tau+C(1+\ln r)\\
  \leq &C r^{2-k_1}(\ln r)^{k_2}\\
\end{array}
$$
for some $C>0$ for all $r>2$.
This yields

$$
|a_{0,1}(r)|\leq Cc_0r^{2-k_1}(\ln r)^{k_2},\quad\forall~r>2.
$$
Similar argument also holds for $k_1>2$ case with the integrate range changed from $(2,r)$ into $(r,+\infty)$. But for $k_1=2$, from \eqref{equ-cancellation} we have no cancellation and it yields

$$
|a_{0,1}(r)|\leq Cc_0r^{2-k_1}(\ln r)^{k_2+2},\quad\forall~r>2.
$$

For $0<k_1<1$ case, we pick $0<\epsilon:=\frac{1}{2}\min\{1,2-2k_1, 2k_1\}$ such that

$$
\left\{
\begin{array}{llll}
 3-2k_1-\epsilon>-1,\\
 1-2k_1-\epsilon>-1,\\
 3-2k-2k_1+\epsilon<-1, & \forall~k\geq 2.\\
\end{array}
\right.
$$
Then by \eqref{Equ-converge} and H\"older inequality, we have

$$
\begin{array}{llll}
&\displaystyle a_{0,1}^2(r)+\sum_{k=1}^{+\infty}\sum_{m=1}^{2}a_{k,m}^2(r)\\
  \leq & \displaystyle  Cc_0^2r^{4-2k_1}(\ln r)^{2k_2}
  +2\sum_{k=1}^{+\infty}\sum_{m=1}^2r^{-2k}\left|
  \int_{2}^r\tau^{1+k}b_{k,m}(\tau)d\tau
  \right|^2\\
  &\displaystyle
  +2\sum_{m=1}^2r^{2}\left|\int_{2}^rb_{1,m}(\tau)d\tau\right|^2
  +2\sum_{k=2}^{+\infty}\sum_{m=1}^2r^{2k}\left|
\int_{+\infty}^r\tau^{1-k}b_{k,m}(\tau)d\tau
\right|^2\\
\leq & \displaystyle Cc_0^2r^{4-2k_1}(\ln r)^{2k_2}\\
&+\displaystyle 2\sum_{k=1}^{+\infty}\sum_{m=1}^2r^{-2k}\int_2^r\tau^{3+2k-2k_1-\epsilon}(\ln\tau)^{2k_2}d\tau
\cdot \int_2^r
\tau^{2k_1}(\ln\tau)^{-2k_2}b_{k,m}^2(\tau)\frac{d\tau}{\tau^{1-\epsilon}}
\\
&+\displaystyle 2\sum_{m=1}^2r^{2}\int_2^r\tau^{1-2k_1-\epsilon}(\ln\tau)^{2k_2}d\tau
\cdot
\int_2^r\tau^{2k_1}(\ln\tau)^{-2k_2}b_{1,m}^2(\tau)\frac{d\tau}{\tau^{1-\epsilon}}
\\
&+\displaystyle
2\sum_{k=2}^{+\infty}\sum_{m=1}^2r^{2k}\int^{+\infty}_r\tau^{3-2k-2k_1+\epsilon}
(\ln\tau)^{2k_2}d\tau
\cdot\int^{+\infty}_r\tau^{2k_1}(\ln\tau)^{-2k_2}b_{k,m}^2(\tau)\frac{d\tau}{\tau^{1+\epsilon}}
\\
\leq & \displaystyle Cc_0^2r^{4-2k_1}(\ln r)^{2k_2}+ Cr^{4-2k_1-\epsilon}(\ln r)^{2k_2}\int_2^r \sum_{k=1}^{+\infty}\sum_{m=1}^2
\tau^{2k_1}(\ln\tau)^{-2k_2}b_{k,m}^2(\tau)\frac{d\tau}{\tau^{1-\epsilon}}
\\
& +\displaystyle Cr^{4-2k_1+\epsilon}(\ln r)^{2k_2}\int_r^{+\infty} \sum_{k=2}^{+\infty}\sum_{m=1}^2
\tau^{2k_1}(\ln\tau)^{-2k_2}b_{k,m}^2(\tau)\frac{d\tau}{\tau^{1+\epsilon}}
\\
\leq & \displaystyle Cc_0^2r^{4-2k_1}(\ln r)^{2k_2}\\
\end{array}
$$
for some $C>0$ relying only on $k_1,k_2$ and $p$.

For $1<k_1<2$ case, we only need to change $a_{1,m}$ with $m=1,2$ into
\begin{equation}\label{equ-temp-10}
a_{1,m}=-\dfrac{1}{2}r\int_{+\infty}^r\tau b_{1, m}(\tau) \mathrm{d} \tau+\frac{1}{2} r^{-1} \int_{2}^{r} \tau^{2} b_{1, m}(\tau) \mathrm{d} \tau,
\end{equation}
and $0<\epsilon:=\frac{1}{2}\min\{1,4-2k_1,2k_1-2\}$. The estimates on $\displaystyle a_{0,1}^{2}(r)+\sum_{k=1}^{+\infty} \sum_{m=1}^{2} a_{k, m}^{2}(r)$ follow similarly.

For $k_1=1$ case, we choose $\epsilon:=\frac{1}{2}$ and we use the following estimates of $a_{1,m}$.

$$
\begin{array}{llll}
a_{1,m}^2(r) &\leq &\displaystyle Cc_0^2r^2\int_2^r\tau^{-1}(\ln\tau)^{2k_2}d\tau\cdot \int_2^r\tau^2 (\ln\tau)^{-2k_2} b_{1,m}^2(\tau)\frac{d\tau}{\tau}\\
&&+\displaystyle Cc_0^2r^{-2}\int_2^{r}\tau^3(\ln\tau)^{2k_2}d\tau\cdot \int_2^r\tau^2 (\ln\tau)^{-2k_2} b_{1,m}^2(\tau)\frac{d\tau}{\tau}\\
&\leq & Cc_0^2 r^2(\ln\tau)^{2k_2+2}.
\end{array}
$$
The rest parts of estimate follow similarly.

For $k_1=2$ case, we choose $\epsilon:=\frac{1}{2}$ and change $a_{1,m}$ with $m=1,2$ into \eqref{equ-temp-10}. In this case, the estimates of $a_{0,1}$ shall be

$$
\begin{array}{llll}
  a_{0,1}^2(r) & \leq & \displaystyle Cc_0^2\int_2^r\tau^{-1}(\ln\tau)^{2+2k_2} d\tau
  \cdot
  \int_2^r\tau^{4}(\ln\tau)^{-2k_2}\frac{d\tau}{\tau}\\
  &&\displaystyle +Cc_0^2(\ln r)^2\int_2^r\tau^{-1}(\ln\tau)^{2k_2}d\tau
  \cdot \int_2^r\tau^4 (\ln\tau)^{-2k_2}b_{0,1}^2(\tau)\frac{d\tau}{\tau}\\
  &\leq & Cc_0^2(\ln r)^{2k_2+4}.
\end{array}
$$
The rest parts of estimate follow similarly.

This proves that $v(r)$ is well-defined, is a solution of \eqref{equ:Laplace} in distribution sense \cite{Gunther-ConformalNormalCoord} and satisfies

$$
||v(r\cdot)||^2_{L^2(\mathbb S^{1})}\leq
\left\{
\begin{array}{lllll}
 Cc_0^2r^{4-2k_1}(\ln r)^{2k_2}, & \text{if }k_1\not\in\mathbb N_*,\\
 Cc_0^2 r^{4-2k_1}(\ln r)^{2k_2+2}, & \text{if }k_1\in\mathbb N_*\setminus\{2\},\\
  Cc_0^2 (\ln r)^{2k_2+4}, & \text{if }k_1=2,\\
\end{array}
\right.
$$
By interior regularity theory of elliptic differential equations,  $v$ is smooth \cite{Book-Gilbarg-Trudinger}.
Then the pointwise decay rate at infinity follows from re-scaling method and weak Harnack inequality as Theorem 8.17 of \cite{Book-Gilbarg-Trudinger} etc. (see also Lemma 3.1 of \cite{Liu-Bao-2021-Dim2}).
\end{proof}
\begin{remark}
  By H\"older inequality, the constant $C$ relying on $p$ in \eqref{equ:FastVanishing}  remains finite when $p=\infty$ in  \eqref{equ:vanishing-g-Lp}. Furthermore, by approximation method, the results hold for more general right hand side term other than smooth functions.
\end{remark}
Similar to Lemma 3.2 in \cite{Liu-Bao-2020-ExpansionSPL}, by interior estimate, we have the following
\begin{lemma}\label{Lem-fastConverge-2}
  Let $g\in C^{\infty}(\mathbb R^2)$ satisfy

  $$
  g=O_l(|x|^{-k_1}(\ln|x|)^{k_2})\quad\text{as }|x|\rightarrow\infty
  $$
  for some $k_1>0,k_2\geq 0$ and $l-1\in\mathbb N$. Then

  $$
  v_g=
  \left\{
  \begin{array}{lllll}
  O_{l+1}(|x|^{2-k_1}(\ln |x|)^{k_2}), &\text{if }k_1\not\in\mathbb N_*,\\
  O_{l+1}(|x|^{2-k_1}(\ln |x|)^{k_2+1}), &\text{if }k_1\in\mathbb N_*\setminus\{2\},\\
  O_{l+1}(|x|^{2-k_1}(\ln |x|)^{k_2+2}), &\text{if }k_1=2,\\
  \end{array}
  \right.\quad \text{as }|x|\rightarrow\infty,
  $$
  where $v_g$ denotes the solution found in Lemma \ref{Lem-fastConverge}.
\end{lemma}

\section{Quadratic part of $u$ at infinity}\label{seclabel-WeakConvergence}

In this section, we prove a weaker asymptotic behavior than \eqref{equ:asymptotic-behav}, which concerns only on the quadratic part of $u$ at infinity.

\begin{lemma}\label{lem:WeakConvergence}
  Let $u,f$ be as in Theorem \ref{thm:perturb-dim2}. Then there exist $A\in\mathtt{Sym}(2)$ satisfying $F_{\tau}(\lambda(A))=f(\infty)$ and $\epsilon>0$ such that
  \begin{equation}\label{equ:weakConverge}
  \left|u(x)-\frac{1}{2}x^TAx\right|=
  O_2(|x|^{2-\epsilon})
  \end{equation}
  as $|x|\rightarrow\infty.$
\end{lemma}

For $\tau=0$, the results in Lemma \ref{lem:WeakConvergence} were proved earlier in Bao--Li--Zhang \cite{Bao-Li-Zhang-ExteriorBerns-MA}.
More rigorously, in the proof of Theorem 1.2 in \cite{Bao-Li-Zhang-ExteriorBerns-MA} (see also Theorem 2.2 of \cite{Liu-Bao-2021-Expansion-LagMeanC}), we have the following result for Monge--Amp\`ere type equations, which holds for general $n\geq 2$.
\begin{theorem}\label{thm:weakConverge-MA}
  Let $u \in C^{0}\left(\mathbb{R}^{n}\right)$ be a convex viscosity solution of

  \begin{equation*}
\operatorname{det} D^{2} u=\psi(x) \quad \text {in } \mathbb{R}^{n}
\end{equation*}
with $u(0)=\min _{\mathbb{R}^{n}} u=0$, where $0<\psi \in C^{0}\left(\mathbb{R}^{n}\right)$ and

\begin{equation*}
\psi^{\frac{1}{n}}-1 \in L^{n}\left(\mathbb{R}^{n}\right).
\end{equation*}
Then there exists a linear transform $T$ satisfying $\operatorname{det} T=1$ such that $v:=u \circ T$ satisfies

$$
  \left|v-\dfrac{1}{2}|x|^2\right|\leq C|x|^{2-\varepsilon},\quad \forall~ |x|\geq 1
  $$
  for some $C>0$ and $\varepsilon>0$.
\end{theorem}

For $\tau\in(0,\frac{\pi}{4}]$ cases, the results follow from Legendre transform and the asymptotic behavior for the Monge--Amp\`ere type equations and Poisson equations. More explicitly, as proved for general $n\geq 2$ cases in Theorem 2.1 and Remark 2.6 of \cite{Liu-Bao-2021-Expansion-LagMeanC}, we have the following result.
\begin{theorem}\label{thm:weakConv-Small}
  Let $u\in C^2(\mathbb R^2)$ be a classical solution of \eqref{Equ-SPL-perturb} with $\tau\in[0,\frac{\pi}{4}]$ and $f\in C^1(\mathbb R^2)$ satisfy
   \begin{equation}\label{1-orderCondition}
|x|^{\zeta}|f(x)-f(\infty)|+|x|^{1+\zeta'}
  \left| Df ( x ) \right|\leq C,\quad\forall~|x|>1
  \end{equation}
for some $C>0$ with $\zeta>1,\zeta'>0$ for $\tau\in[0,\frac{\pi}{4})$
 and $\zeta,\zeta'>0$ for $\tau=\frac{\pi}{4}$. Let $D^2u$ and $f(\infty)$ satisfy
  \eqref{equ:cond:semi-convex} and  \eqref{equ:cond:fInfinity} respectively.
  For $\tau\in(0,\frac{\pi}{4}]$, we assume further that $u$ satisfies \eqref{equ:cond:quadratic-growth}.
Then there exist $\epsilon>0, A\in\mathtt{Sym}(2)$ satisfying $F_{\tau}(\lambda(A))=f(\infty)$ and \eqref{equ:cond:semi-convex} and $C>0$ such that

$$
  ||D^2u||_{C^{\alpha}(\mathbb R^2)}\leq C
  \quad\text{and}\quad \left|D^2u(x)-A\right|\leq \dfrac{C}{|x|^{\epsilon}},\quad\forall~|x|>1.
  $$
\end{theorem}

\begin{proof}[Proof of Lemma \ref{lem:WeakConvergence} with $0\leq\tau\leq \frac{\pi}{4}$]
In Lemma \ref{lem:WeakConvergence}, we have $f\in C^m(\mathbb R^2)$ satisfies \eqref{Low-Regular-Condition} for some $\zeta>2$ and $m\geq 3$. Especially $f\in C^1(\mathbb R^2)$ satisfies condition \eqref{1-orderCondition} with $\zeta'=\zeta>2$. By Theorem \ref{thm:weakConv-Small} we have $||D^2u||_{C^{\alpha}(\mathbb R^2)}$ is bounded for all $0<\alpha<1$ and $D^2u(x)$ converge to matrix $A$ at H\"older speed $|x|^{-\epsilon}$. By Newton--Leibnitz formula, since $Du, u$ are bounded on $\partial B_1$, for any $|x|>1$ we let $e:=\frac{x}{|x|}\in\partial B_1$ to obtain

$$
\begin{array}{llllll}
 \displaystyle \left|Du(x)-Du(e)-\int_0^{|x|-1}Ads\cdot e\right| & = &\displaystyle \left|\int_0^{|x|-1}D^2u((s+1)e)-Ads\cdot e\right|\\
  &\leq & \displaystyle C|x|^{1-\epsilon}.\\
\end{array}
$$
Consequently there exists $C>0$ such that

$$
\left|Du(x)-Ax\right|\leq C|x|^{1-\epsilon},\quad\forall~|x|> 1.
$$
By Newton--Leibnitz formula again we have

$$
\begin{array}{llll}
  \displaystyle \left|
  u(x)-u(e)-\int_0^{|x|-1}x^TA\cdot eds
  \right| & = & \displaystyle
  \left|
  \int_0^{|x-1|}\left(Du((s+1)e)-Ax\right)\cdot eds
  \right|\\
  &\leq & C|x|^{2-\epsilon},\quad\forall~|x|> 1
\end{array}
$$
for some $C>0$. This finishes the proof of Lemma \ref{lem:WeakConvergence} with $\tau\in[0,\frac{\pi}{4}]$.
\end{proof}

\begin{proof}[Proof of Lemma \ref{lem:WeakConvergence} with $\frac{\pi}{4}<\tau<\frac{\pi}{2}$]
By semi-convex condition \eqref{equ:cond:semi-convex}, we have $\lambda_i>-(a+b)$ for $i=1,2$ and consequently by a direct computation as in \cite{Huang-Wang-SelfShrinkingSol,Warren-Calibrations-MA},

$$
  \arctan \frac{\lambda_{i}+a-b}{\lambda_{i}+a+b}=\arctan \frac{\lambda_{i}+a}{b}-\frac{\pi}{4}.
$$
Consequently Eq.~\eqref{Equ-SPL-perturb} with $\tau\in(\frac{\pi}{4},\frac{\pi}{2})$ and semi-convex condition $D^2u>-(a+b)I$ becomes

$$
\arctan\dfrac{\lambda_1(D^2u)+a}{b}+
\arctan\dfrac{\lambda_2(D^2u)+a}{b}=\dfrac{b}{\sqrt{a^2+1}}f(x)+\frac{\pi}{2}\quad\text{in }\mathbb R^2.
$$
Let $v:=\frac{1}{b}(u+\frac{a}{2}|x|^2)$, then $v$ satisfies \eqref{equ:cond:quadratic-growth} for some new $C>0$ and the Lagrangian mean curvature equation
\begin{equation}\label{equ:temp}
\arctan\lambda_1(D^2v)+\arctan\lambda_2(D^2v)=\dfrac{b}{\sqrt{a^2+1}}f(x)+\frac{\pi}{2}\quad\text{in }\mathbb R^2.
\end{equation}
If $\frac{b}{\sqrt{a^2+1}}f(\infty)\not=-\frac{\pi}{2}$, we
 may assume without loss of generality that $\frac{b}{\sqrt{a^2+1}}f(\infty)+\frac{\pi}{2}> 0$, otherwise we consider the equation satisfied by $-v$ as replacement. Then the desired result \eqref{equ:weakConverge} follows from $\tau=\frac{\pi}{2}$ case, which will be proved below.

It remains to prove for $\frac{b}{\sqrt{a^2+1}}f(\infty)+\frac{\pi}{2}=0$ case. For any sufficiently small $\delta>0$, there exists $R_0>0$ such that

$$
\left|\dfrac{b}{\sqrt{a^2+1}}f(x)+\frac{\pi}{2}\right|<\delta,\quad\forall ~ |x|>R_0.
$$
Since $D^2v>-I$, together with the continuity of $D^2u$ in $\mathbb R^2$, we have $D^2v$ bounded on entire $\mathbb R^2$. Consequently $D^2u$ is bounded on entire $\mathbb R^2$. For any sufficiently large $|x|>1$, we set

$$
R:=|x|\quad\text{and}\quad u_R(y):=\dfrac{4}{R^2}u(x+\frac{R}{2}y)\quad\text{in }B_1.
$$
Then $u_R$ is a classical solution of

$$
F_{\tau}(\lambda(D^2u_R(y)))=f(x+\frac{R}{2}y)=:f_R(y)\quad\text{in }B_1.
$$
Since $D^2u_R$ are uniformly (to $R$) bounded, the equations above are uniformly elliptic. Consequently, by the definition of $u_R$,
we have $||u_R||_{C^0(B_1)}$ are also uniformly bounded.
By interior H\"older estimates for second derivatives as in Theorem 17.11 of \cite{Book-Gilbarg-Trudinger}, we have
\begin{equation}\label{equ-temp-8}
||D^2u_R||_{C^{\alpha}(B_{\frac{1}{2}})}\leq C
\end{equation}
for some $\alpha,C>0$ uniform to sufficiently large $R$. Now we turn to the algebraic form of \eqref{equ:temp} i.e.,

\begin{equation}\label{equ-temp-9}
\Delta v=(1-\det D^2v)\cdot \tan\left(\frac{b}{\sqrt{a^{2}+1}} f(x)+\frac{\pi}{2}\right)=:g(x)
\end{equation}
in $\mathbb R^2$.
By approximation and Lemma \ref{Lem-fastConverge}, there exists a solution $v_g$ solving \eqref{equ-temp-9} on $\mathbb R^2\setminus\overline{B_1}$ with estimate

$$
|v_g(x)|\leq C|x|^{2-\zeta}(\ln|x|),\quad\forall~|x|>1
$$
for some  $C>0$. Let
$
v_R(y):=v_g(x+\frac{R}{2}y)
$ in $B_1$, where $R=|x|>2$, then

$$
\Delta v_R=g(x+\frac{R}{2}y)=:g_R(y)\quad\text{in }B_1.
$$
By conditions \eqref{Low-Regular-Condition} and \eqref{equ-temp-8}, we have

$$
||v_R||_{C^0(B_1)}\leq CR^{-\zeta}\ln R\quad\text{and}\quad
||g_R||_{C^{\alpha}(B_{\frac{1}{2}})}\leq CR^{\alpha-\zeta}
$$
for some $C>0$ uniform to $R>2$. By interior Schauder estimates, we have

$$
||v_R||_{C^{2,\alpha}(B_{\frac{1}{3}})}\leq CR^{-\epsilon'}
$$
for some $0<\alpha,\epsilon'<1$ and $C>0$.
Consequently there exists $C>0$ such that

$$
v_g=O_2(|x|^{2-\epsilon'})\quad\text{as }|x|\rightarrow\infty.
$$
Since $v-v_g$ is harmonic on   $\mathbb R^2\setminus\overline{B_1}$ with bounded Hessian matrix, by spherical harmonic expansion as in   \eqref{Equ-def-Wronski}, there exists $A_v\in\mathtt{Sym}(2),\beta_v\in\mathbb R^2$ and $c_v,d_v\in\mathbb R$ such that

$$
v-v_g=\frac{1}{2}x^TA_vx+\beta_v\cdot x+d_v\ln|x|+c_v+O_l(|x|^{-1})
$$
for all $l\in\mathbb N$
as $|x|\rightarrow\infty$. Combining the two asymptotic behavior above, we have

$$
\begin{array}{llll}
u(x) & = & bv(x)-\frac{a}{2}|x|^2\\
&=& \frac{1}{2}x^T(bA_v-aI)x+b\beta_v\cdot x+bd_v\ln|x|+bc_v+bv_g+O_2(|x|^{-1})\\
&=&  \frac{1}{2}x^T(bA_v-aI)x+O_2(|x|^{2-\epsilon'})\\
\end{array}
$$
as $|x|\rightarrow\infty$.
This finishes the proof of \eqref{equ:weakConverge}.
\end{proof}
\begin{proof}[Proof of Lemma \ref{lem:WeakConvergence} with $\tau=\frac{\pi}{2}$]
  Consider the algebraic form of Eq.~\eqref{Equ-SPL-perturb} with $\tau=\frac{\pi}{2}$ i.e.,

$$
\cos f(x)\cdot \Delta u+\sin f(x)\det D^2u=\sin f(x)
$$
in $\mathbb R^2$. By condition  \eqref{equ:cond:fInfinity},  $\cot f(\infty)\not=0$ and consequently we have

$$
\det D^2u+\cot f(\infty)\cdot \Delta u=1+\left(\cot f(\infty)-\cot f(x)\right)\Delta u
$$
in $\mathbb R^2$. Change of variable by setting

$$
v(x):=u(x)+\frac{\cot f(\infty)}{2}|x|^2,
$$
which satisfies
\begin{equation}\label{equ:transformed}
\begin{array}{lll}
\det D^2v & = & \displaystyle (\lambda_1(D^2u)+\cot f(\infty))\cdot
(\lambda_2(D^2u)+\cot f(\infty))\\
&=& \displaystyle 1+\cot^2f(\infty)+
\left(
\cot  f(\infty)-\cot f(x)
\right)\Delta u\\
&=:& g(x)\\
\end{array}
\end{equation}
in $\mathbb R^2$. To obtain the desired results, we shall obtain the asymptotic behavior of $g(x)$ at infinity  and apply Theorem \ref{thm:weakConverge-MA}.

\textbf{Step 1:} We prove the boundedness of $D^2u$ by interior Hessian estimate.

\noindent Since $f(\infty)\in(0,\pi)$, for any sufficiently small $0<\delta<f(\infty)$, there exists $R_0>0$ such that

$$
|f(x)-f(\infty)|<\delta,\quad\forall~|x|>R_0.
$$
By Eq.~\eqref{Equ-SPL-perturb} with $\tau=\frac{\pi}{2}$, for all $i=1,2$,  we have

$$
\arctan\lambda_i+\frac{\pi}{2}>\arctan\lambda_1+\arctan\lambda_2>f(\infty)-\delta>0,
\quad\forall~|x|>R_0.
$$
Consequently by the monotonicity of $\arctan$ function, we have

$$
D^2u>-\cot(f(\infty)-\delta)I
\quad\forall~|x|>R_0.
$$
By \eqref{equ:cond:quadratic-growth} and the quadratic growth condition from above, there exists $C>0$ such that
\begin{equation}\label{temp-1}
|u(x)|\leq C(1+|x|^2),\quad\forall~x\in\mathbb R^2.
\end{equation}
For sufficiently large $|x|>2R_0$, we set
\begin{equation}\label{equ-scaled-1}
R:=|x|>2R_0\quad\text{and}\quad u_R(y):=\dfrac{4}{R^2}u(x+\frac{R}{2}y),\quad y\in B_1,
\end{equation}
where $B_r(x)$ denote the ball centered at $x$ with radius $r$ and $B_r:=B_r(0)$.
Then $u_R\in C^4(B_1)$ satisfies
\begin{equation}\label{equ-scaled-2}
\arctan\lambda_1(D^2u_R)+\arctan\lambda_2(D^2u_R)=f(x+\frac{R}{2}y)=:f_R(y)\quad\text{in }B_1.
\end{equation}
By a direct computation, \eqref{temp-1} implies that there exists a constant $C$ uniform to $R>2R_0$ such that

$$
\sup_{B_1}|u_R|\leq \dfrac{4}{R^2}\sup_{B_{\frac{|x|}{2}}(x)}|u|\leq C.
$$
Together with condition \eqref{Low-Regular-Condition} on $f$ with $\zeta>2$ and $m\geq 3$, we have

$$
\begin{array}{lllll}
||f_R||_{C^{1,1}(B_1)} & = & \displaystyle
\sup_{B_1}\left(|f_R|+|Df_R|\right)+\sup_{y,z\in B_1\atop y\not=z}\dfrac{|Df_R(y)-Df_R(z)|}{|y-z|}\\
&\leq & \displaystyle \sup_{B_1}\left(|f_R|+|Df_R|+|D^2f_R|\right)\\
&\leq & CR^{-\zeta},\\
\end{array}
$$
for some constant $C>0$ uniform to $R>2R_0$. Furthermore, since

$$
\sup_{B_1}\left|f_R(y)-f(\infty)\right|\rightarrow 0\quad\text{as }R\rightarrow\infty.
$$
By $f(\infty)>0$, there exists $\delta>0$ uniform to sufficiently large $R$ such that $|f_R|\geq \delta>0$.

Now we introduce the following interior Hessian estimates for Lagrangian mean curvature equations as in Theorems 1.1 and 1.2 by Bhattacharya \cite{Bhattacharya-HessianEstiamte-LagMeanCur}.
\begin{theorem}\label{thm:HessianEst-1}
  Let $u$ be a $C^4$ solution of \eqref{equ:SPL} in $B_r\subset\mathbb R^n$, where $f\in C^{1,1}(B_r)$ and
  \begin{equation}\label{equ:supercritical}
  |f|\geq \frac{n-2}{2}\pi+\delta
  \end{equation}
  for some $\delta>0$. Then we have

  \begin{equation*}
\left|D^{2} u(0)\right| \leq C_{1} \exp \left(\frac{C_2}{r^{2 n-2}} \left(\underset{B_r(0)}{osc}\frac{u}{r}+1\right)^{2n-2}\right),
\end{equation*}
where $C_{1}, C_{2}$ are positive constants depending on $\|f\|_{C^{1,1}\left(B_{r}\right)}, n$, and $\delta$.
\end{theorem}
Applying Theorem \ref{thm:HessianEst-1}  to the equation satisfies by $u_R$ in $B_1$, there exists $C>0$ uniform to $R>2R_0$ such that

$$
|D^2u_R(0)| \leq C_1\exp\left(C_2 \left(\underset{B_1(0)}{osc}u_R+1\right)^2\right)\leq C.
$$
Consequently $D^2u$ is bounded on entire $\mathbb R^2$.

\textbf{Step 2:} Now we compute the asymptotic behavior of $g(x)$ at infinity.

\noindent By the definition of $g(x)$ in Eq.~\eqref{equ:transformed} and condition \eqref{Low-Regular-Condition}, since $D^2u$ is bounded on entire $\mathbb R^2$, we have

$$
\begin{array}{llll}
  g(x)-(1+\cot^2f(\infty))& =& (\cot f(\infty)-\cot f(x))\Delta u\\
  &\leq & \dfrac{C}{\sin^2f(\infty)} |f(x)-f(\infty)|\\
  &\leq & C|x|^{-\zeta}\\
\end{array}
$$
for some constant $C>0$. Consequently there exists $C>0$ such that

$$
\begin{array}{lllll}
& \displaystyle \int_{\mathbb R^2\setminus\overline{B_1}}\left|
g^{\frac{1}{2}}(x)-(1+\cot^2f(\infty))^{\frac{1}{2}}
\right|^2dx\\
\leq &\displaystyle C\int_{\mathbb R^2\setminus\overline{B_1}}|g(x)-(1+\cot^2f(\infty))|^2dx\\
\leq &\displaystyle
C\int_{\mathbb R^2\setminus\overline{B_1}}|x|^{-2\zeta}d x\\
< &\infty.\\
\end{array}
$$

\textbf{Step 3:} Obtain the asymptotic behavior of $v$.

\noindent
Since $1+\cot^2f(\infty)>0$, by the continuity of $D^2v$ and $n=2$, we have either $D^2v>0$ or $D^2v<0$ for sufficiently large $|x|>2R_0$. We may assume without loss of generality that $D^2v>0$, otherwise we consider $-v$ instead. Thus by extension result as Theorem 3.2 of \cite{Min-Extension-ConvexFunc}, we can apply Theorem \ref{thm:weakConverge-MA} (see also Corollary 2.3 in \cite{Liu-Bao-2021-Expansion-LagMeanC})  after re-scaling $\widetilde v:=\frac{1}{(1+\cot^2f(\infty))^{\frac{1}{2}}}v$, there exist $A_v\in\mathtt{Sym}(2)$ satisfying $\det A_v=1+\cot^2f(\infty)$ such that

$$
\left|
v-\dfrac{1}{2}x^TA_vx
\right|\leq C|x|^{2-\epsilon}\quad\forall~|x|>1
$$
for some $C>0$ and $\epsilon>0$. By the definition of $v$, we have

$$
\left|
u-\frac{1}{2}x^T(A_v-\cot f(\infty)I)x
\right|\leq C|x|^{2-\epsilon}\quad\forall~|x|>1.
$$
By taking $A:=A_v-\cot f(\infty)I$, it is easy to verify that

$$
\cos f(\infty)  \operatorname{trace}(A)+\sin f(\infty)\det A=\sin f(\infty),
$$
and hence $\arctan\lambda_1(A)+\arctan\lambda_2(A)=f(\infty)$ and the first part of the desired result in Lemma \ref{lem:WeakConvergence} follows immediately.

\textbf{Step 4:}  We prove interior gradient and Hessian estimates by scaling.

\noindent Let $u_R, f_R$ be as in \eqref{equ-scaled-1} and \eqref{equ-scaled-2}. From the results in Step 1, $D^2u$ is bounded on entire $\mathbb R^2$. Consequently equations \eqref{equ-scaled-2} are uniformly (to $R>2R_0$) elliptic. By interior H\"older estimates for second derivatives as in Theorem 17.11 of \cite{Book-Gilbarg-Trudinger}, we have

$$
[D^2u_R]_{C^{\alpha}(B_{\frac{1}{2}})}\leq C
$$
for some $\alpha,C>0$, where $\alpha$ relies only on $||D^2u_R||_{C^0(B_1)}$ and $C$ relies only on $||u_R||_{C^2(B_1)}$ and $
||f_R||_{C^1(B_1)}$. Consequently

$$
||D^2u_R||_{C^{\alpha}(B_{\frac{1}{2}})}<C
$$
for some $C>0$ for all $R>2R_0$.
We would like to mention that  $n=2$ is necessary to apply interior estimates as  Theorem 17.11 in \cite{Book-Gilbarg-Trudinger}. For higher dimensions, it is generally required that the operator has a concavity structure (see for instance Theorem 17.14 in \cite{Book-Gilbarg-Trudinger} and Theorem 8.1 in \cite{Book-Caffarelli-Cabre-FullyNonlinear}).

  To obtain the desired result, it remains to prove the gradient and Hessian estimate on the difference between $u$ and $\frac{1}{2}x^TAx$. Let
\begin{equation}\label{equ-scaled-3}
w(x):=u(x)-\frac{1}{2}x^TAx\quad\text{and}\quad
w_R(y):=\frac{4}{R^2}w(x+\frac{R}{2}y)\quad\forall~ y\in B_1.
\end{equation}
From the results in Steps 1-4, there exists $C>0$ uniform to all  $R>2R_0$ such that

$$
||u_R||_{L^{\infty}(B_1)}\leq C\quad\text{and}\quad ||w_R||_{L^{\infty}(B_1)}\leq CR^{-\epsilon}.
$$
Applying Newton--Leibnitz formula between \eqref{equ-scaled-2} and $\arctan\lambda_1(A)+\arctan\lambda_2(A)=f(\infty)$, we have
\begin{equation}\label{equ-scaled-5}
a_{ij}^RD_{ij}w_R=f_R(y)-f(\infty)\quad\text{in }B_1,
\end{equation}
where

$$
a_{ij}^R(y)=\int_0^1D_{M_{ij}}F_{\tau}(A+tD^2w_R(y))dt,
$$
are uniformly elliptic and having uniformly bounded (to $R>2R_0$) $C^{\alpha}$ norm.
Hereinafter, we let $[a_{ij}]$ denote the 2-by-2 matrix with the $i,j$-position being $a_{ij}$ and $D_{M_{ij}}F_{\tau}(M)$ denote the value of partial derivative of $F_{\tau}(\lambda(M))$ with respect to $M_{ij}$ variable at $M=[M_{ij}]$.
By condition \eqref{Low-Regular-Condition}, there exists uniform $C>0$ such that
\begin{equation}\label{equ-temp-1}
\left\|f_{R}-1\right\|_{L^{\infty}\left(B_{1}\right)}+\sum_{k=1}^{m-1}\left\|D^{k} f_{R}\right\|_{C^{\alpha}\left(B_{1}\right)} \leq C R^{-\zeta}.
\end{equation}
By interior Schauder estimates as Theorem 6.2 in \cite{Book-Gilbarg-Trudinger},

$$
||w_R||_{C^{2,\alpha}(B_{\frac{1}{2}})}\leq C\left(
||w_R||_{L^{\infty}(B_1)}+||f_{R}-f(\infty)||_{C^{\alpha}(\overline{B_1})}
\right)\leq CR^{-\min\{\epsilon,\zeta\}}.
$$
This finishes the proof of Lemma \ref{lem:WeakConvergence} by choosing $\epsilon$ as the minimum of $\epsilon$ and $\zeta$.
\end{proof}

\section{Asymptotic behavior and expansions of $u$ at infinity}\label{seclabel-AsyExpan}

In this section, we prove the asymptotic behavior and expansions of solution $u$ at infinity following the line of iteration method as by Bao--Li--Zhang \cite{Bao-Li-Zhang-ExteriorBerns-MA} with an improvement from spherical harmonic expansion.

\begin{lemma}\label{lem:Initial}
  Let $u,f$ be as in Theorem \ref{thm:perturb-dim2} and

  $$
  w(x):=u(x)-\frac{1}{2}x^TAx,
  $$
  where $A\in\mathtt{Sym}(2)$ is from Lemma \ref{lem:WeakConvergence}. Then there exist $C,\alpha,\epsilon'>0$ such that
  \begin{equation}\label{Chap6-Asy-Initial}
  \left\{ \begin{array} { l } { \left| D ^ { k } w ( x ) \right| \leq C |x | ^ { 2 - k - \epsilon' } , } \\ { \frac { \left| D ^ { m + 1 } w \left( x _ { 1 } \right) - D ^ { m + 1 } w \left(x _ { 2 } \right) \right| } { \left| x _ { 1 } - x _ { 2 } \right| ^ { \alpha } } \leq C \left| x_ { 1 } \right| ^ { 1 - m - \epsilon' - \alpha },} \end{array} \right.
\end{equation}
for all  $ | x | > 2, k = 0 , \ldots , m + 1$ and $\left| x_ { 1 } \right| > 2,x _ { 2 } \in B _ { \left| x _ { 1} \right| / 2 } \left( x _ { 1 } \right)$.
\end{lemma}
\begin{proof}
  For sufficiently large $|x|>1$, we set $R,u_R,f_R,w_R$ as in the proof of Lemma \ref{lem:WeakConvergence} i.e., \eqref{equ-scaled-1}, \eqref{equ-scaled-2} and \eqref{equ-scaled-3}. As proved in Lemma \ref{lem:WeakConvergence}, together with $u\in C^2(\mathbb R^2)$ we have

  $$
  |D^kw(x)|\leq C|x|^{2-k-\epsilon},\quad\forall~|x|>1
  $$
  for some $C>0$ and $\epsilon>0$ from Lemma \ref{lem:WeakConvergence} for $k=0,1,2$. It remains to prove the higher order derivatives following the  Step 4 in the proof of Lemma \ref{lem:WeakConvergence}.

  In fact, we consider the scaled equation
  \begin{equation}\label{equ-scaled-4}
  F_{\tau}(\lambda(D^2u_R))=f_R(y)\quad\text{in }B_1.
  \end{equation}
  Since $D^2u_R$ are uniformly (to $R$) bounded, $F_{\tau}(\lambda(D^2u_R))$
  are uniformly elliptic. By Theorem 17.11 in \cite{Book-Gilbarg-Trudinger} there exist $\alpha,C,R_0>0$ such that
  \begin{equation}\label{equ-temp-2}
  ||D^2u_R||_{C^{\alpha}(B_{\frac{2}{3}})}\leq C,\quad\forall~R>2.
  \end{equation}
  Apply Newton--Leibnitz formula between
  \eqref{equ-scaled-4} and $F_{\tau}(\lambda(A))=f(\infty)$ to obtain
  linearized equation \eqref{equ-scaled-5}. By
  the boundedness of $D^2u_R$ and
  estimate \eqref{equ-temp-2}, the coefficients $a_{ij}^R(y)$ are uniformly elliptic and having finite $C^{\alpha}$-norm for some $\alpha>0$ uniform to $R>2$. Together with \eqref{equ-temp-1}, by interior Schauder estimates we have

$$
||w_R||_{C^{2,\alpha}(B_{\frac{1}{2}})}\leq C\left(
||w_R||_{L^{\infty}(B_{\frac{2}{3}})}+||f_{R}-f(\infty)||_{C^{\alpha}(\overline{B_{\frac{2}{3}}})}
\right)\leq CR^{-\min\{\epsilon,\zeta\}}
$$
for some $C>0$ uniform to all $R>2$.

It remains to obtain higher order derivative estimates.  For any $e\in\partial B_1$, we act partial derivative $D_e$ to both sides of \eqref{equ-scaled-4}
  and obtain

  $$
  \widetilde a_{ij}^RD_{ij}(D_ew_R)=D_ef_R(y)\quad\text{in }B_1,\quad\text{where}\quad
  \widetilde a_{ij}^R(y):=D_{M_{ij}}F_{\tau}(D^2u_R(y)).
  $$
By the boundedness of $D^2u_R$ and estimate \eqref{equ-temp-2}, the coefficients $\widetilde a_{ij}^R$ are uniformly elliptic with uniformly bounded $C^{\alpha}$-norm to all $R>2$. By interior Schauder estimate again and the arbitrariness of $e\in\partial B_1$, we have

\begin{equation*}
\left\|w_{R}\right\|_{C^{3, \alpha}\left(\overline{B}_{1 / 3}\right)} \leq C\left(\left\|w_{R}\right\|_{L^{\infty}\left(\overline{B}_{2/3}\right)}+| D f_{1, R} \|_{C^{\alpha}\left(\overline{B}_{2/ 3}\right)}\right) \leq C R^{-\min\{\epsilon,\zeta\}}.
\end{equation*}
By taking further derivatives of Eq.~\eqref{equ-scaled-4}, higher order derivatives follow from Schauder estimate and this finishes the proof of \eqref{Chap6-Asy-Initial}.
\end{proof}

Next we prove an iteration type lemma that improves the estimates in Lemma \ref{lem:Initial}. The result follows as in Lemma 2.2 in Bao--Li-Zhang \cite{Bao-Li-Zhang-ExteriorBerns-MA} with minor modifications.
\begin{lemma}\label{lem:iteration}
  Let $u\in C^2(\mathbb R^2)$ be a classical solution of
  \begin{equation}\label{equ-generalFull}
    F(D^2u)=f(x)\quad\text{in }\mathbb R^2.
  \end{equation}
  Suppose $F$ is smooth up to the boundary of the range of $D^2u$ and is uniformly elliptic. Suppose further that $f\in C^m(\mathbb R^2)$ satisfies \eqref{Low-Regular-Condition} for some $\zeta>2,m\geq 3$. If there exists a quadratic polynomial $v$ satisfying $F(D^2v)=f(\infty)$ and $w:=u-v$ satisfies
  \begin{equation}\label{Chap6-equ-converge1}
  \left\{ \begin{array} { l } { \left| D ^ { k } w ( x ) \right| \leq C | x | ^ { 2 - \epsilon - k }(\ln|x|)^{p_0} , } \\ { \frac { \left| D ^ { m + 1 } w \left( x _ { 1 } \right) - D ^ { m + 1 } w \left( x_ { 2 } \right) \right| } { \left| x _ { 1 } - x _ { 2 } \right| ^ { \alpha } } \leq C \left| x _ { 1 } \right| ^ { 1 - m - \epsilon - \alpha }(\ln|x|)^{p_0}} \end{array} \right.
  \end{equation}
  for some $0<\epsilon<\frac{1}{2}$ and $p_0\in\mathbb N$
  for all  $ | x | > 2, k = 0 , \ldots , m + 1$ and $\left| x_ { 1 } \right| > 2,x _ { 2 } \in B _ { \left| x _ { 1} \right| / 2 } \left( x _ { 1 } \right)$.
  Then
   \begin{equation}\label{Chap6-equ-converge2}
  \left\{ \begin{array} { l } { \left| D ^ { k } w ( x ) \right| \leq C | x | ^ { 2 - 2 \epsilon - k }
  (\ln|x|)^{2p_0}
  ,} \\ { \frac { \left| D ^ { m + 1 } w \left( x _ { 1 } \right) - D ^ { m + 1 } w \left( x _ { 2 } \right) \right| } { \left| x _ { 1 } - x _ { 2 } \right| ^ { \alpha } } \leq C \left| x _ { 1 } \right| ^ { 1 - m - 2 \epsilon - \alpha }
  (\ln|x|)^{2p_0}
  ,} \end{array} \right.
  \end{equation}
for all  $ | x | > 2, k = 0 , \ldots , m + 1$ and $\left| x_ { 1 } \right| > 2,x _ { 2 } \in B _ { \left| x _ { 1} \right| / 2 } \left( x _ { 1 } \right)$.
\end{lemma}

\begin{proof}
Acting partial derivative $D_k$ to both sides of Eq.~ \eqref{equ-generalFull}, we have
\begin{equation}\label{chap6-equ-diff-1}
a_{ij}(x)D_{ij}(D_ku(x))=D_kf(x),\quad\text{where}\quad
a_{ij}(x)=D_{M_{ij}}F(D^2u(x)).
\end{equation}
By the assumptions on $F$ and $u$, $a_{ij}(x)$ are uniformly elliptic coefficients. By the first formula of  \eqref{Chap6-equ-converge1}, since $m\geq 3$, there exists $C>0$ such that

$$
|a_{ij}(x)-D_{M_{ij}}F(A)|\leq C|x|^{-\epsilon}(\ln|x|)^{p_0},\quad
|Da_{ij}(x)|\leq C|x|^{-1-\epsilon}(\ln|x|)^{p_0}
$$
for all $|x|>2$.
For the given $\alpha\in(0,1)$, together with the second formula in \eqref{Chap6-equ-converge1} we have

$$
\frac { \left| D a _ { i j } \left( x _ { 1 } \right) - D a _ { i j } \left( x _ { 2 } \right) \right| } { \left| x _ { 1 } - x _ { 2 } \right| ^ { \alpha } } \leq C \left| x _ { 1 } \right| ^ { - 1 - \epsilon - \alpha }(\ln|x_1|)^{p_0} , \quad \left| x _ { 1 } \right| > 2, \quad x _ { 2 } \in B _ { \left| x _ { 1 } \right| / 2 } \left( x _ { 1 } \right).
$$

For any $l=1,2$, we act partial derivative $D_l$ to both sides of \eqref{chap6-equ-diff-1}. Let $h_1:=D_{kl}u$, then we have

$$
D_{M_{ij},M_{qr}}F(D^2u)D_{ijk}uD_{qrl}u+
D_{M_{ij}}F(D^2u)D_{ij}h_1=D_{kl}f(x),
$$
i.e.,

$$
a_{ij}(\infty)D_{ij}h_1=f_2(x)\quad\text{in }\mathbb R^2,
$$
where

$$
f_2(x):=
D_{kl}f-D_{M_{ij},M_{qr}}F(D^2u)D_{ijk}uD_{qrl}u+
\left(a_{ij}(\infty)-a_{ij}(x)\right)D_{ij}h_1.
$$
Notice that $[a_{ij}(\infty)]=[D_{M_{ij}}F(A)]$ is a positive symmetric matrix, we set $Q:=[a_{ij}(\infty)]^{\frac{1}{2}}$ and  $\widetilde h_1(x):=h_1(Qx)$. Since trace is invariant under cyclic permutations,
\begin{equation}\label{Chap6-equ-temp-Poisson1}
 \Delta\widetilde h_1(x)=f_2(Qx)=:\widetilde f_2(x)
\end{equation}
in $Q^{-1}(R^2\setminus\overline{B_1})$.
Since $Q$ is invertible, by $\zeta>2>2\epsilon$ and   \eqref{Chap6-equ-converge1}, we have

$$
\begin{array}{lllll}
|\widetilde f_2(x)| &\leq & |D_{kl}f(Qx)|+C|D_{ijk}u(Qx)|\cdot |D_{qkl}u(Qx)|+C|a_{ij}(\infty)-a_{ij}(Qx)|\\
&\leq & C|x|^{-2-\zeta}+C|x|^{-2-2\epsilon}(\ln|x|)^{2p_0}\\
&\leq &C|x|^{-2-2\epsilon}(\ln|x|)^{2p_0}\\
\end{array}
$$
in $x\in Q^{-1}(\mathbb R^2\setminus\overline{B_1})$. By Lemmas \ref{Lem-fastConverge}
and  \ref{Lem-fastConverge-2} and  $0<2\epsilon<1$,  there exists a function $\widetilde h_2$ satisfying \eqref{Chap6-equ-temp-Poisson1} on an exterior domain with estimate

$$
|D^k\widetilde h_2(x)|\leq C|x|^{-2\epsilon-k}(\ln|x|)^{2p_0},\quad\forall~k=0,1.
$$
By the definition of $\tilde h_1$ and the first line of \eqref{Chap6-equ-converge1},  $\widetilde h_1-\widetilde h_2$ is harmonic on exterior domain of $\mathbb R^2$ with vanishing speed

$$
\widetilde h_1-\widetilde h_2-\delta_{kl}=O(|x|^{-\epsilon}(\ln|x|)^{p_0})\quad\text{as }|x|\rightarrow\infty.
$$
By spherical harmonic expansion as in \eqref{Equ-def-Wronski}, see also formula (2.23) in \cite{Bao-Li-Zhang-ExteriorBerns-MA}, there exists $C>0$ such that

$$
|\widetilde h_1-\widetilde h_2-\delta_{kl}|\leq C|x|^{-1}
$$
for all $|x|>2$. By taking $h_2(x):=\widetilde h_2(Q^{-1}x)$, there exists $C>0$ such that

$$
\left|h_1(x)-h_2(x)-A_{kl}\right|\leq C|x|^{-1},\quad\forall~|x|>2
$$
and hence

$$
|D^kw(x)|\leq C|x|^{2-2\epsilon-k}(\ln|x|)^{2p_0},\quad\forall~|x|>2.
$$
By taking higher order derivatives as in the proof of Lemma \ref{lem:Initial}, this finishes the proof of the first formula in \eqref{Chap6-equ-converge2}.

It remains to prove the H\"older semi-norm part. In fact for sufficiently large
$|x|$, we set $R:=|x|$ and

\begin{equation*}
h_{2, R}(y):=\widetilde h_{2}\left(x+\frac{R}{4} y\right), \quad f_{2, R}(y)=\frac{R^{2}}{16} \widetilde f_{2}\left(x+\frac{R}{4} y\right), \quad|y| \leq 2.
\end{equation*}
By condition \eqref{Low-Regular-Condition} on $f(x)$ and   \eqref{Chap6-equ-converge1},
for any large $|x_1|$ and   $x_2\in B_{|x_1|/2}(x_1)$ with $x_1\not=x_2$, we have

$$
\begin{array}{lllll}
  &\dfrac{\left|\widetilde f_{2}\left(x_{1}\right)-\widetilde{f_{2}}\left(x_{2}\right)\right|}{\left|x_{1}-x_{2}\right|^{\alpha}}\\ \leq &
  \dfrac{\left|\widetilde D_{kl}f\left(Qx_{1}\right)-D_{kl}f\left(Qx_{2}\right)\right|}{\left|x_{1}-x_{2}\right|^{\alpha}}
\\  &+
  \dfrac{\left|\left(a_{i j}(\infty)-a_{i j}(Qx_1)\right) D_{i j} h_{1}\left(Qx_{1}\right)-\left(a_{i j}(\infty)-a_{i j}(Qx_2)\right) D_{i j} h_{1}\left(Qx_{2}\right)\right|}{\left|x_{1}-x_{2}\right|^{\alpha}}\\
  &+
  \dfrac{\left|F_{M_{i j}, M_{q r}}\left(D^{2} u\right) D_{i j k} u D_{q r l} u\left(Qx_{1}\right)-F_{M_{i j}, M_{q r}}\left(D^{2} u\right) D_{i j k} u D_{q r l} u\left(Qx_{2}\right)\right|}{\left|x_{1}-x_{2}\right|^{\alpha}}\\
  \leq & C|x_1|^{-\zeta-2-\alpha}+C|x_1|^{-1-\epsilon}(\ln|x_1|)^{p_0}\cdot |x_1|^{-1-\epsilon-\alpha}(\ln|x_1|)^{p_0}\\
  \leq & C|x_1|^{-2-2\epsilon-\alpha}(\ln|x_1|)^{2p_0}.\\
\end{array}
$$
Thus by a direct computation, there exists $C>0$ uniform to all $R>1$ such that

$$
||f_{2,R}||_{C^{\alpha}(\overline{B_1})}\leq CR^{-2\epsilon}(\ln|x|)^{2p_0}.
$$
By the interior Schauder estimates of Poisson equation, we have

\begin{equation*}
\left\|h_{2, R}\right\|_{C^{2, \alpha}\left(B_{1}\right)} \leq C\left(\left\|h_{2, R}\right\|_{L^{\infty}\left(B_{2}\right)}+\left\|f_{2, R}\right\|_{C^{\alpha}\left(B_{2}\right)}\right) \leq C R^{-2 \epsilon}(\ln R)^{2p_0}
\end{equation*}
By the definition of $h_{2,R}$ and the non-degeneracy of $Q$ we have the second formula in  \eqref{Chap6-equ-converge2}.

For higher order derivatives, the results follow from taking further derivatives to both sides of the equation and apply interior Schauder estimates as in the proof of Lemma \ref{lem:Initial}.
\end{proof}

Now we are able to prove the asymptotic behavior at infinity by iteration. More explicitly, we have the following results.
\begin{proposition}
  Let $u,f$ be as in Theorem \ref{thm:perturb-dim2}, then there exists $A\in\mathtt{Sym}(2)$ satisfying $F_{\tau}(\lambda(A))=f(\infty)$ and \eqref{equ:cond:semi-convex}, $b\in\mathbb R^2$ and $d,c,R_1\in\mathbb R$
  such that
  \eqref{equ:asymptotic-behav} holds as $|x|\rightarrow\infty$, where $Q$ is given by \eqref{equ:def-matrixQ}. Furthermore, when $\zeta>3$, then there exist $d_1,d_2$ such that \eqref{equ:asymptotic-expan} holds.
\end{proposition}
\begin{proof}
  By Lemmas \ref{lem:WeakConvergence} and \ref{lem:Initial}, there exist $\alpha,\epsilon'>0$ such that \eqref{Chap6-Asy-Initial} holds. Let $p_1\in\mathbb N$ be the positive integer such that

  $$
  2^{p_1}\epsilon'<1\quad\text{and}\quad 1<2^{p_1+1}\epsilon'<2.
  $$
  (If necessary, we may choose $\epsilon'$ smaller to make both inequalities hold.)
  Let $\epsilon_1:=2^{p_1}\epsilon'$.
  Applying Lemma \ref{lem:iteration} $p_1$ times, we have
  \begin{equation}\label{Chap6-equ-converge5}
  \left\{ \begin{array} { l } { \left| D ^ { k } w ( x ) \right| \leq C | x | ^ { 2 - \epsilon_1 - k }  , } \\ { \frac { \left| D ^ { m + 1 } w \left( x _ { 1 } \right) - D ^ { m + 1 } w \left( x_ { 2 } \right) \right| } { \left| x _ { 1 } - x _ { 2 } \right| ^ { \alpha } } \leq C \left| x _ { 1 } \right| ^ { 1 - m - \epsilon_1 - \alpha},} \end{array} \right.
  \end{equation}
  for all  $ | x | > 2, k = 0 , \ldots , m + 1$ and $\left| x_ { 1 } \right| > 2,x _ { 2 } \in B _ { \left| x _ { 1} \right| / 2 } \left( x _ { 1 } \right)$.

  Now we consider the linearized equation again. Applying Newton--Leibnitz formula between Eq.~\eqref{Equ-SPL-perturb} and $F_{\tau}(\lambda(A))=f(\infty)$, we have $\widetilde a_{ij}D_{ij}w=f(x)-f(\infty)$ i.e.,

  $$
  \widetilde a_{ij}(\infty)D_{ij}w=f(x)-f(\infty)+(\widetilde a_{ij}(\infty)-\widetilde a_{ij}(x))D_{ij}w=:f_3(x)
  $$
  in $\mathbb R^2$,
  where the coefficients are uniformly elliptic and

  $$
  \widetilde a_{ij}(x)=\int_0^1D_{M_{ij}}F_{\tau}(A+tD^2w(x))dt\rightarrow \widetilde a_{ij}(\infty)=D_{M_{ij}}F_{\tau}(A)
  $$
  as $|x|\rightarrow\infty$. Let $Q:=[D_{M_{ij}}F(A)]^{\frac{1}{2}}$ and $\widetilde w(x):=w(Qx)$. By the invariance of trace under cyclic permutations again, we have
  \begin{equation}\label{equ-temp-3}
  \Delta\widetilde w=f_3(Qx)=:\widetilde f_3(x).
  \end{equation}
  By the definition of $\widetilde a_{ij}(x)$, condition \eqref{Low-Regular-Condition} on $f$ and \eqref{Chap6-equ-converge5} we have

  $$
  |\widetilde f_3(x)|\leq C|x|^{-2\epsilon_1}
  $$
  for some $C,R_1>0$ for all $|x|>2R_1$. By  Lemmas \ref{Lem-fastConverge}
and  \ref{Lem-fastConverge-2},  there exist a function $\widetilde h_3$ solving \eqref{equ-temp-3} in $\mathbb R^2\setminus\overline{B_{R_1}}$ with estimate

  $$
  |\widetilde h_3(x)|\leq C|x|^{2-2\epsilon_1}
  $$
  for some  $C>0$ for all $|x|>2R_1$. Thus $\widetilde w-\widetilde h_3$ is harmonic on an exterior domain of $\mathbb R^2$ with $\widetilde w-\widetilde h_3=O(|x|^{2-\epsilon_1})$ as $|x|\rightarrow\infty$. By spherical harmonic expansion as in \eqref{Equ-def-Wronski} or the proof of (2.31) in \cite{Bao-Li-Zhang-ExteriorBerns-MA}, there exist $\widetilde b\in\mathbb R^2$ and $\widetilde d_1,\widetilde d_2\in\mathbb R$ such that

  $$
  \widetilde w(x)-\widetilde h_3(x)=\widetilde b\cdot x+\widetilde d_1\ln|x|+\widetilde d_2+O(|x|^{-1})
  $$
as $|x|\rightarrow\infty$   and consequently

  $$
  \begin{array}{lllll}
  |\widetilde w(x)-\widetilde b\cdot x| &\leq & |\widetilde h_3(x)| +|\widetilde d_1\ln|x|+\widetilde d_2+O(|x|^{-1})|\\
  &=& O(|x|^{2-2\epsilon_1} ) +|\widetilde d_1\ln|x|+\widetilde d_2+O(|x|^{-1})|\\
  &=& O(|x|^{2-2\epsilon_1} )=o(|x|)\\
  \end{array}
  $$
  as $|x|\rightarrow\infty$.

  Let

  $$
  \widetilde w_1(x):=\widetilde w(x)-\widetilde b\cdot x.
  $$
  By interior estimates as used in Lemma \ref{lem:WeakConvergence}, we have

  $$
  |D^k\widetilde w_1(x)|\leq C|x|^{2-2\epsilon_1-k} ,
  $$
  for some $C>0$ for all $k=0,\cdots,m+1$ and $|x|>2R_1$.
  As in the process in obtaining \eqref{equ-temp-3}, $\widetilde w_1$ satisfies
  \begin{equation}\label{equ-temp-4}
  \Delta \widetilde w_1=\widetilde f_4(x)=O(|x|^{-\zeta})+O(|x|^{-4\epsilon_1}).
  \end{equation}
Since $\zeta>2$ and $4\epsilon_1\in (2,4)$,  Lemmas \ref{Lem-fastConverge}
and  \ref{Lem-fastConverge-2},  there exists a function $\widetilde h_4$ solving \eqref{equ-temp-4} in $\mathbb R^2\setminus\overline{B_{R_1}}$ with estimate

$$
|\widetilde h_4(x)|\leq
\left\{
\begin{array}{llll}
C|x|^{2-\zeta}(\ln|x|)+C|x|^{2-4\epsilon_1}, & \zeta\not\in\mathbb N_*,\\
C|x|^{2-\zeta}+C|x|^{2-4\epsilon_1}, & \zeta\in\mathbb N_*,\\
\end{array}
\right.
$$
for some $C>0$ for all $|x|>2R_1$. Thus $\widetilde w_1-\widetilde h_4$ is harmonic in $|x|>2R_1$ with $|\widetilde w_1-\widetilde h_4|=O(|x|^{2-2\epsilon_1})$. Since $2-2\epsilon_1<1$, by spherical harmonic expansion, there exist $\widetilde d,\widetilde d_3\in\mathbb R$ such that

$$
\widetilde w_1(x)-\widetilde h_4(x)=\widetilde d\ln|x|+\widetilde d_3+O(|x|^{-1})
$$
as $|x|\rightarrow\infty$ and consequently,
\begin{equation}\label{equ-temp-6}
\begin{array}{llllll}
  |\widetilde w_1(x)-\widetilde d\ln|x|| & \leq & |\widetilde h_4(x)|+\widetilde d_3+O(|x|^{-1})\\
  &=& \left\{
  \begin{array}{lll}
    O(|x|^{2-\zeta}), & \text{if }\zeta<4\epsilon_1, & \text{and }\zeta\not=3\\
    O(|x|^{2-\zeta}(\ln|x|)), & \text{if }\zeta<4\epsilon_1, & \text{and }\zeta=3\\
    O(|x|^{2-4\epsilon_1}), & \text{if }\zeta\geq 4\epsilon_1.\\
  \end{array}
  \right.\\
\end{array}
\end{equation}
Again, we follow the process in obtaining \eqref{equ-temp-3}.

 Since

$$
F_{\tau}(\lambda(D^2(\frac{1}{2}x^TAx+d\ln|x|)))=f(\infty)+O(|x|^{-4})
$$
as $|x|\rightarrow\infty$, we set

$$
\widetilde w_2(x):=\widetilde w_1(x)-\widetilde d\ln|x|,
$$
which satisfies

$$
  \Delta \widetilde w_2=O(|x|^{-\zeta})+O(|x|^{-4})+\left\{
  \begin{array}{lll}
    O(|x|^{-2\zeta}), & \text{if }\zeta<4\epsilon_1, & \text{and }\zeta\not=3\\
    O(|x|^{-2\zeta}(\ln|x|)^2), & \text{if }\zeta<4\epsilon_1, & \text{and }\zeta=3\\
    O(|x|^{-8\epsilon_1}), & \text{if }\zeta\geq 4\epsilon_1.\\
  \end{array}
  \right.
$$
Since $8\epsilon_1\in(4,8)$, we have

$$
\Delta \widetilde w_2=O(|x|^{-\zeta})+O(|x|^{-4}).
$$
By Lemmas \ref{Lem-fastConverge} and \ref{Lem-fastConverge-2}, we have a solution $\widetilde h_5$ on exterior domain with estimate

$$
|\widetilde h_5(x)|\leq
\left\{
\begin{array}{llll}
C|x|^{2-\zeta}+C|x|^{-2}(\ln|x|), & \text{if }\zeta\not\in\mathbb N_*,\\
C|x|^{2-\zeta}(\ln|x|)+C|x|^{-2}(\ln|x|), & \text{if }\zeta\in\mathbb N_*,\\
\end{array}
\right.
$$
for some $C>0$ for all $|x|>2R_1$. Together with \eqref{equ-temp-6},  by spherical harmonic expansion we have  $\widetilde c\in\mathbb R$ such that
\begin{equation}\label{equ-temp-7}
\widetilde w_2(x)-\widetilde h_5(x)=\widetilde c+O(|x|^{-1}).
\end{equation}
Rotating back by $Q^{-1}$ matrix, since $P=Q^{-2}$,  we have $\beta\in\mathbb R^2, c,d\in\mathbb R$ such that

$$
\begin{array}{llll}
&\displaystyle \left|
u(x)-\left(\frac{1}{2}x^TAx+\beta x+d\ln(x^TPx)+c\right)
\right|\\
\leq&  C|\widetilde w_2-\widetilde c|\\
\leq & C|\widetilde h_5(Q^{-1}x)|+C|x|^{-1}\\
\leq &
\left\{
\begin{array}{llll}
C|x|^{2-\zeta}+C|x|^{-1},& \text{if }\zeta\not=3,\\
C|x|^{2-\zeta}(\ln|x|)+C|x|^{-1},& \text{if }\zeta=3,\\
\end{array}
\right.\\
\end{array}
$$
for some $C>0$ for sufficiently large $|x|$. Estimates for higher order
derivatives follow similarly as in Lemma \ref{lem:Initial}.  The second equality in \eqref{equ:def-matrixQ} can be obtained in (1.4) of \cite{Liu-Bao-2021-Dim2}.
This finishes the proof of \eqref{equ:asymptotic-behav}.

It remains to prove that when $\zeta>3$, we have \eqref{equ:asymptotic-expan} at infinity. In fact from \eqref{equ-temp-7}, we iterate once more by setting
$\widetilde w_3:=\widetilde w_2(x)-\widetilde c$, which satisfies

$$
\Delta\widetilde w_3=O(|x|^{-\zeta})+O(|x|^{-4})\quad\text{as }|x|\rightarrow\infty.
$$
By Lemmas \ref{Lem-fastConverge} and \ref{Lem-fastConverge-2}, we have a solution $\widetilde h_6$ on exterior domain with estimate

$$
|\widetilde h_6(x)|\leq
\left\{
\begin{array}{lllll}
C|x|^{2-\zeta}+C|x|^{-2}(\ln|x|), & \text{if }\zeta\not\in\mathbb N_*,\\
C|x|^{2-\zeta}(\ln|x|)+C|x|^{-2}(\ln|x|), & \text{if }\zeta\in\mathbb N_*,\\
\end{array}
\right.
$$
for some $C>0$ in $|x|>2R_1$. Since $\widetilde w_3-\widetilde h_6$ is harmonic on an exterior domain of $\mathbb R^2$ and satisfies $|\widetilde w_3-\widetilde h_6|
=O(|x|^{-1})
$ as $|x|\rightarrow\infty$, by spherical harmonic expansion there exists $\widetilde d_4,\widetilde d_5\in\mathbb R$ such that

$$
\widetilde w_3-\widetilde h_6=\widetilde d_4\cos\theta |x|^{-1}+\widetilde d_5 \sin\theta |x|^{-1}+O(|x|^{-2})
$$
as $|x|\rightarrow\infty$, where $\theta=\frac{x}{|x|}$ here. By rotating back through $Q^{-1}$, we have the 0-order estimates in \eqref{equ:asymptotic-expan} since $P=Q^{-2}$. For higher order derivatives, the result follows from interior estimate as in \cite{Liu-Bao-2021-Expansion-LagMeanC}.
\end{proof} 

\small

\bibliographystyle{plain}

\bibliography{C:/Bib/Thesis}

\bigskip

\noindent Z. Liu \& J. Bao

\medskip

\noindent  School of Mathematical Sciences, Beijing Normal University\\
Laboratory of Mathematics and Complex Systems, Ministry of Education\\
Beijing 100875, China \\[1mm]
Email: \textsf{liuzixiao@mail.bnu.edu.cn, jgbao@bnu.edu.cn}

\end{document}